\documentclass[12pt,reqno]{amsart}
\usepackage[T1]{fontenc}
\usepackage{lmodern}
\usepackage{cite,amssymb,amsfonts,amsmath,amsthm,mathrsfs,dsfont,microtype}
\usepackage[dvipsnames]{xcolor}

\usepackage[shortlabels]{enumitem}
\usepackage[a4paper, margin=2.54cm]{geometry}

\usepackage{setspace}
\usepackage{hyperref}
\hypersetup{
	colorlinks=true,
	linkcolor=blue,
	citecolor=blue,
	urlcolor=blue,
	pdfauthor={
		Sandro Coriasco,
		Alexandre Kirilov,
		Wagner A. A. de Moraes,
		Pedro M. Tokoro
	},
	pdftitle={
		Top-Degree Global Solvability for 
		Tube Complexes in Gevrey Ultradistributions
	}
}

\numberwithin{equation}{section}
\allowdisplaybreaks[2]

\theoremstyle{plain}
\newtheorem{theorem}{Theorem}[section]
\newtheorem{proposition}[theorem]{Proposition}
\newtheorem{lemma}[theorem]{Lemma}
\newtheorem{corollary}[theorem]{Corollary}

\theoremstyle{definition}
\newtheorem{definition}[theorem]{Definition}

\theoremstyle{remark}
\newtheorem{remark}[theorem]{Remark}

\newcommand{\R}{{\mathbb{R}}} 
\newcommand{\C}{{\mathbb{C}}} 
\newcommand{\Z}{{\mathbb{Z}}} 
\newcommand{\N}{{\mathbb{N}}} 
\newcommand{\T}{{\mathbb{T}}} 
\renewcommand{\L}{{\mathbb{L}}} 
\newcommand{\EE}{{\mathbb{E}}}
\newcommand{\FF}{{\mathbb{F}}}

\newcommand{\G}{{\mathscr{G}}} 
\newcommand{\D}{{\mathscr{D}}} 
\newcommand{\F}{{\mathscr{F}}} 

\newcommand{\df}{{\mathsf{\Lambda}}} 

\begin{document}
	
	\title[Top-Degree Global Solvability]{Top-Degree Global Solvability for Tube Complexes in Gevrey Ultradistributions}
	
	\author[S. Coriasco]{Sandro Coriasco}
	\address{
		Dipartimento di Matematica ``Giuseppe Peano'', 
		Università Degli Studi di Torino, 
		Via Carlo Alberto 10, CAP 10123, Torino, 
		Italia }
	\email{sandro.coriasco@unito.it}
		
	\author[A. Kirilov]{Alexandre Kirilov}
	\address{
		Departamento de Matem\'atica,
		Universidade Federal do Paran\'a,  
		Caixa Postal 19096, CEP 81531-980, Curitiba, Paran\'a,
		Brasil}
	\email{akirilov@ufpr.br}
	
	\author[W. de Moraes]{Wagner A. A. de Moraes}
	\address{
		Departamento de Matem\'atica,  
		Universidade Federal do Paran\'a  
		Caixa Postal 19096, CEP 81531-980, Curitiba, Paran\'a,
		Brasil}
	\email{wagnermoraes@ufpr.br}
	
	\author[P. Tokoro]{Pedro M. Tokoro}
	\address{
		Programa de P\'os-Gradua\c c\~ao em Matem\'atica,  
		Universidade Federal do Paran\'a,  
		Caixa Postal 19096, CEP 81531-980, Curitiba, Paran\'a,  
		Brasil}
	\email{pedro.tokoro@ufpr.br}
	
	


		\begin{abstract}
			Let \(s>1\), let \(M\) be a connected, non-compact, oriented
			real-analytic manifold, and let
			\(\omega_1,\ldots,\omega_m\) be real-valued closed \(1\)-forms of
			Gevrey order \(s\) on \(M\). We study the differential complex naturally
			associated with this family on \(M\times\T^m\). We prove that its
			top-degree operator is globally solvable in Roumieu Gevrey
			ultradistributions, or equivalently that the corresponding top-degree
			cohomology vanishes. No global hypoellipticity assumption and no
			arithmetic condition on the periods of the defining forms are required.
			The proof is carried out in the physical variables and combines fiber
			translations, a local normal form, and a transport formula along paths in
			the base manifold. These tools yield propagation of Gevrey regularity,
			non-confinement of Gevrey singularities, and the support control needed
			to apply an abstract solvability criterion. The result highlights a
			sharp contrast with the compact setting, where compatibility conditions
			are unavoidable and solvability for compatible data may depend on
			exponential small-denominator conditions.
		\end{abstract}

	\maketitle
	
	{
	\small  
	\setstretch{1} 
	\tableofcontents 
	}

\section{Introduction}
\label{sec:introduction}

The relation between regularity and solvability is a classical theme in
the theory of partial differential equations. At the local level,
hypoellipticity of an operator often leads, by transposition, to
solvability of its formal transpose. In a global setting, however, this
relation is more delicate: topology, behavior at infinity, support
properties, and small denominators may all affect the existence and
regularity of global solutions.

A natural setting in which these phenomena interact is that of tube
structures associated with closed differential forms. Let \(M\) be a
real-analytic manifold of dimension \(n\), fix \(s>1\), and let
\[
\omega_1,\ldots,\omega_m\in\df^1\G^s(M)
\]
be a family of real-valued closed Gevrey \(1\)-forms. On
\[
X:=M\times\T^m
\]
consider the differential operators
\begin{equation}\label{eq:intro-complex}
	\L^q
	=
	\mathrm d_t
	+
	\sum_{k=1}^m
	\omega_k\wedge\partial_{x_k},
	\qquad
	q=0,\ldots,n-1,
\end{equation}
where \(\mathrm d_t\) denotes the exterior derivative in the
\(M\)-variable. Since the forms \(\omega_k\) are closed and the
operators \(\partial_{x_k}\) commute, we have
\[
\L^{q+1}\L^q=0.
\]

Thus, \(\L=(\L^q)_{q=0}^{n-1}\) defines a differential complex.
Locally, it is generated by a system of commuting first-order vector
fields and belongs to the general framework of locally integrable and
involutive structures; see \cite{BCH_book,Treves}.

Global regularity and solvability for complexes of the form
\eqref{eq:intro-complex} have been studied extensively on compact
manifolds. For a single torus variable, Bergamasco, Cordaro, and
Malagutti \cite{BCM1993} investigated global hypoellipticity of the first
operator and exhibited the role of arithmetic properties of the defining
closed form. The corresponding top-degree solvability problem was
studied by Bergamasco, Cordaro, and Petronilho \cite{BCP1996}, using the
identification of the last operator with the transpose of the first one.
On the torus, Bergamasco and Petronilho \cite{BP1999_jmaa} considered
solvability in every degree and obtained conditions involving rational
and Liouville forms. Subsequent works developed smooth, Gevrey, and
cohomological aspects of tube structures of higher corank; see, among
others,
\cite{DM2016,DM2020,ADL2023gh,ADL2023gs,AFJR2024}.

On a compact base, unrestricted surjectivity in top degree is impossible:
the constant functions belong to the kernel of \(\L^0\) and therefore
impose compatibility conditions on the range of \(\L^{n-1}\). Even for
compatible data, solvability may depend on exponential
small-denominator conditions; see \cite{ADL2023gs}. Thus, in the compact
setting, regularity and solvability remain closely connected with the
arithmetic properties of the periods of
\[
\boldsymbol{\omega}
:=
(\omega_1,\ldots,\omega_m).
\]

The passage to a non-compact base changes the problem substantially.
Global hypoellipticity of \(\L^0\) on certain non-compact manifolds still
depends on arithmetic properties of
\(\boldsymbol{\omega}\); see \cite{CKMT}. Analogous phenomena occur in
the Gevrey setting, where the relevant small-denominator estimates must
reflect Gevrey growth; see \cite{CKMTGevrey2026}. The purpose of the
present paper is to show that top-degree global solvability is markedly
more robust.

Our main result states that, if \(M\) is connected, non-compact, and
oriented, then
\begin{equation}\label{eq:intro-main-map}
	\L^{n-1}:
	\df^{0,n-1}\D_s'(M\times\T^m)
	\longrightarrow
	\df^{0,n}\D_s'(M\times\T^m)
\end{equation}
is globally solvable in Gevrey ultradistributions. In other words, every
top-degree Gevrey ultradistribution is of the form
\(\L^{n-1}u\) for some
\[
u\in\df^{0,n-1}\D_s'(M\times\T^m).
\]
Equivalently, the top-degree cohomology of the Gevrey
ultradistribution complex vanishes.

No global hypoellipticity assumption on \(\L^0\) and no Diophantine or
exponential-Liouville condition on the periods of the forms
\(\omega_k\) is required. This is the main distinction from global
Gevrey regularity and from compatible solvability on compact bases.
The non-compactness of \(M\) provides the support properties that make
unrestricted top-degree solvability possible.

The proof combines geometric arguments in the physical variables with
an abstract functional-analytic criterion. Locally, the closed forms
\(\omega_k\) admit Gevrey potentials. Translation of the torus fibers by
these potentials conjugates the tube complex to the de Rham differential
in the base variable, giving a local normal form. A transport formula
along paths in \(M\) then shows that Gevrey regularity on one full torus
fiber propagates throughout the connected base. For compactly supported
functions, non-compactness supplies a fiber on which the function
vanishes; this yields both injectivity and non-confinement of Gevrey
singularities for \(\L^0\). The same transport mechanism also gives the
support control required for \(\L^{n-1}\)-convexity.

The functional-analytic framework is provided by
Ara\'ujo \cite{Araujo2017}: injectivity and non-confinement for the
transpose yield semiglobal solvability, while convexity with respect to
supports promotes semiglobal solvability to global solvability. Since
the transpose of \(\L^{n-1}\) is, up to sign, \(\L^0\), the geometric
properties established above give the surjectivity of
\eqref{eq:intro-main-map}.

The fiber-translation and transport arguments used here may be viewed as
a physical-space reformulation of the partial Fourier series approach to
global solvability and global hypoellipticity for tube structures; see
\cite{BCM1993,BCP1996,BP1999_jmaa,ADL2023gh,ADL2023gs,
	HZ2017,HZ2019,CKMT}. However, the proof of
our solvability theorem does not involve uniform estimates in the
Fourier frequencies. This explains why the arithmetic conditions that
govern global regularity do not enter the statement of the main result.

The paper is organized as follows. In
Section~\ref{sec:gevrey-sections}, we introduce Gevrey sections and
ultradistributions on real-analytic manifolds and recall the abstract
solvability criterion used later. In Section~\ref{sec:tube-complex}, we
define the tube complex and identify its formal transpose. In
Section~\ref{sec:global-solvability}, we establish the covariance under
fiber translations, the local normal form, and the transport formula.
We then prove propagation of Gevrey regularity, non-confinement of
Gevrey singularities, and propagation of supports, and use these results
to obtain global solvability in top degree. We conclude by comparing the
non-compact result with the arithmetic obstructions that occur on compact
bases.

	\section{Gevrey sections and ultradistributions on analytic manifolds}
	\label{sec:gevrey-sections}

Throughout the paper, all manifolds are assumed to be finite-dimensional,
Hausdorff, and second countable. In particular, they are paracompact and
admit compact exhaustions.

In this section, \(X\) denotes a real-analytic manifold of dimension \(d\),
and \(\EE\) and \(\FF\) denote real-analytic complex vector bundles over
\(X\).

We fix \(s>1\) and work with Gevrey classes of Roumieu type. Our
presentation of Gevrey functions and sections follows \cite{Rod_Gevrey},
whereas the locally convex structure of the corresponding spaces, their
duality properties, and the functional-analytic results used below are
largely based on \cite{Araujo2017}; see also the foundational works of
Komatsu \cite{Komatsu1967,Komatsu_ultra_1}.

\subsection{Local characterization of Gevrey sections} \

Let \(\mathcal U\subset\R^d\) be open, let \(K\Subset\mathcal U\) be
compact, and let \(h>0\). For \(f\in C^\infty(\mathcal U)\), set
\begin{equation}\label{eq:local-gevrey-norm}
	\|f\|_{K,h}
	:=
	\sup_{\alpha\in\N_0^d}
	h^{-|\alpha|}(\alpha!)^{-s}
	\sup_{y\in K}|\partial^\alpha f(y)|.
\end{equation}

The Roumieu Gevrey space of order \(s\) on \(\mathcal U\) is
\begin{equation}\label{eq:local-gevrey-space}
	\G^s(\mathcal U)
	:=
	\left\{
	f\in C^\infty(\mathcal U):
	\forall K\Subset\mathcal U,\ \exists h>0
	\text{ such that }\|f\|_{K,h}<\infty
	\right\}.
\end{equation}

The same notation will be used for vector-valued functions, with
\eqref{eq:local-gevrey-norm} replaced by the sum of the corresponding
componentwise norms.

We recall that \(\G^s(\mathcal U)\) is an algebra and is stable under
differentiation and composition. More precisely, if
\(\Phi:\mathcal U\to\mathcal V\) is a Gevrey map of order \(s\) and
\(f\in\G^s(\mathcal V)\), then \(f\circ\Phi\in\G^s(\mathcal U)\). 
Multiplication by a Gevrey matrix-valued function also preserves Gevrey
regularity. These properties follow from the Leibniz rule and the
Faà di Bruno formula; see \cite[Chapter~1]{Rod_Gevrey}.

Choose a countable, locally finite analytic atlas
\(\{(\Omega_i,\chi_i)\}_{i\in I}\) of \(X\). After refining it if
necessary, fix on each \(\Omega_i\) an analytic frame
\[
\{\mathrm e_{i1},\dots,\mathrm e_{iN}\}
\]
for \(\EE\), where \(N=\operatorname{rank}\EE\). Every section
\(u\in C^\infty(X;\EE)\) can be written on \(\Omega_i\) as
\begin{equation}\label{eq:local-frame-decomposition}
	u|_{\Omega_i}
	=
	\sum_{j=1}^N u_{ij}\,\mathrm e_{ij},
	\qquad
	u_{ij}\in C^\infty(\Omega_i).
\end{equation}

\begin{definition}\label{def:gevrey-section}
	A section \(u\in C^\infty(X;\EE)\) is a Gevrey section of order \(s\)
	if
	\[
	u_{ij}\circ\chi_i^{-1}
	\in\G^s\bigl(\chi_i(\Omega_i)\bigr),
	\qquad
	i\in I,\quad j=1,\dots,N.
	\]
	The space of all such sections is denoted by \(\G^s(X;\EE)\).
\end{definition}

This definition is independent of the chosen analytic atlas and local
frames. Indeed, on the intersection of two trivializing neighborhoods, the
corresponding coefficient vectors are related by analytic changes of
coordinates and analytic transition matrices, both of which preserve the
class \(\G^s\).

\begin{remark}\label{rem:s-equal-one}
	Definition~\ref{def:gevrey-section} also makes sense for \(s=1\), in
	which case it gives the space of real-analytic sections. In this work,
	however, we are concerned with the non-quasianalytic case \(s>1\).
	This assumption guarantees the existence of nontrivial compactly
	supported Gevrey sections, which will serve as test sections in the
	definition of Gevrey ultradistributions.
\end{remark}

\subsection{Compactly supported sections and their topology} \

Let \(\{K_\ell\}_{\ell\in\N_0}\) be a compact exhaustion of \(X\) such that
\begin{equation}\label{eq:compact-exhaustion}
	K_\ell\subset\operatorname{int}K_{\ell+1},
	\qquad
	X=\bigcup_{\ell\in\N_0}K_\ell.
\end{equation}

For each \(\ell\), choose a finite subset \(I_\ell\subset I\) and compact
sets \(K_{\ell,i}\Subset\Omega_i\), with \(i\in I_\ell\), such that
\[
K_\ell\subset\bigcup_{i\in I_\ell}K_{\ell,i}.
\]

For a section \(u\) written as in
\eqref{eq:local-frame-decomposition}, define
\begin{equation}\label{eq:bundle-gevrey-norm}
	\|u\|_{\ell,h,\EE}
	:=
	\sum_{i\in I_\ell}\sum_{j=1}^N
	\bigl\|u_{ij}\circ\chi_i^{-1}\bigr\|_
	{\chi_i(K_{\ell,i}),h}.
\end{equation}

Different choices of finite coverings, analytic charts, and analytic
frames may modify the parameter \(h\), but they lead to the same
locally convex topology introduced below.

For \(\ell\in\N_0\) and \(h>0\), set
\begin{equation}\label{eq:compact-step}
	\G_c^{s,h}(K_\ell;\EE)
	:=
	\left\{
	u\in C^\infty(X;\EE):
	\operatorname{supp}u\subset K_\ell,\ 
	\|u\|_{\ell,h,\EE}<\infty
	\right\}.
\end{equation}

Endowed with the norm \eqref{eq:bundle-gevrey-norm},
\(\G_c^{s,h}(K_\ell;\EE)\) is a Banach space. We define
\begin{equation}\label{eq:fixed-support-inductive-limit}
	\G_c^s(K_\ell;\EE)
	:=
	\underset{h\to+\infty}{\operatorname{ind\,lim}}\,
	\G_c^{s,h}(K_\ell;\EE)
\end{equation}
and
\begin{equation}\label{eq:global-test-space}
	\G_c^s(X;\EE)
	:=
	\underset{\ell\to+\infty}{\operatorname{ind\,lim}}\,
	\G_c^s(K_\ell;\EE).
\end{equation}

Since the positive integers are cofinal in \((0,\infty)\), the inductive
limit in \eqref{eq:fixed-support-inductive-limit} may equivalently be
taken over \(h\in\N\).

As a vector space,
\[
\G_c^s(X;\EE)
=
\bigcup_{\ell\in\N_0}\G_c^s(K_\ell;\EE).
\]
For an arbitrary compact set \(K\subset X\), we use
\(\G_c^s(K;\EE)\) to denote the space
\[
\left\{
u\in\G_c^s(X;\EE):
\operatorname{supp}u\subset K
\right\},
\]
endowed with the analogous fixed-support inductive-limit topology.

\begin{proposition}\label{prop:topology-test-sections}
	The following properties hold.
	\begin{enumerate}[(i)]
		\item
		If \(0<h<h_+\), then the inclusion
		\(\G_c^{s,h}(K_\ell;\EE) \hookrightarrow \G_c^{s,h_+}(K_\ell;\EE)\)
		is compact.
		
		\item
		For every compact set \(K\subset X\), the space
		\(\G_c^s(K;\EE)\) is a DFS space.
		
		\item
		The space \(\G_c^s(X;\EE)\) is a DFS space, and its topology is
		independent of the compact exhaustion, analytic atlas, and local
		frames used in its construction.
		
		\item
		For every compact set \(K\subset X\), the natural inclusion
		\(\G_c^s(K;\EE) \hookrightarrow \G_c^s(X;\EE)\) 
		is continuous and has closed range.
	\end{enumerate}
\end{proposition}

\begin{proof}
	In local analytic frames, the statements reduce to the corresponding
	properties of finite products of scalar Gevrey spaces. The compactness
	in \emph{(i)} follows from the strict gain in the Gevrey parameter and
	a diagonal Arzelà--Ascoli argument. Consequently, the fixed-support
	space in \emph{(ii)} is a DFS space.
	
	For the global space, the directed family
	\[
	\left\{
	\G_c^{s,h}(K_\ell;\EE):
	\ell\in\N_0,\ h>0
	\right\}
	\]
	admits a cofinal sequence whose inclusion maps are compact. Its
	inductive limit is precisely \(\G_c^s(X;\EE)\), which proves
	\emph{(iii)}. The independence of the auxiliary choices follows from
	the invariance of the Gevrey class under analytic changes of
	coordinates and analytic transition matrices. Finally, \emph{(iv)}
	follows from the fixed-support closed-range property; see
	\cite[Lemma~3.3]{Araujo2017}.
\end{proof}

The space \(\G^s(X;\EE)\) is endowed with its usual projective local
topology. A sequence \(\{u_\nu\}\) converges to \(u\) in
\(\G^s(X;\EE)\) if and only if, for every \(\ell\in\N_0\), there exists
\(h_\ell>0\) such that
\[
\|u_\nu-u\|_{\ell,h_\ell,\EE}
\longrightarrow0
\]
and the norms on the left-hand side are finite for all sufficiently large
\(\nu\).

This characterization of convergence will suffice for our purposes. For a
more detailed account of the locally convex structure and duality theory of
Gevrey spaces, we refer to
\cite{Rod_Gevrey,Araujo2017,Komatsu1967,Komatsu_ultra_1}.

\subsection{Gevrey ultradistributions} \

Fix a positive real-analytic density \(\mu\) on \(X\), and let
\(\EE^*\) denote the dual bundle of \(\EE\). Multiplication by \(\mu\)
identifies compactly supported sections of \(\EE^*\) with compactly
supported sections of
\(\EE^*\otimes\mathcal D_X\), where \(\mathcal D_X:=|\Lambda^dT^*X|\)
is the density bundle of \(X\).

We define the space of \(\EE\)-valued Gevrey ultradistributions of order
\(s\) by
\begin{equation}\label{eq:gevrey-ultradistributions}
	\D_s'(X;\EE)
	:=
	\bigl(\G_c^s(X;\EE^*)\bigr)'_b,
\end{equation}
where the subscript \(b\) denotes the strong dual topology.

If \(u\in L_{\mathrm{loc}}^1(X;\EE)\), then \(u\) defines a Gevrey
ultradistribution \(T_u\) by
\begin{equation}\label{eq:regular-ultradistribution}
	\langle T_u,\varphi\rangle
	:=
	\int_X
	\langle u(x),\varphi(x)\rangle_x\,\mu(x),
	\qquad
	\varphi\in\G_c^s(X;\EE^*),
\end{equation}
where \(\langle\cdot,\cdot\rangle_x\) denotes the natural duality pairing
between \(\EE_x\) and \(\EE_x^*\).

\begin{proposition}\label{prop:basic-ultradistribution-properties}
	The space \(\D_s'(X;\EE)\) is an FS space. Moreover, there are
	natural continuous injections
	\begin{equation}\label{eq:natural-inclusions}
		\G^s(X;\EE)
		\hookrightarrow
		C^\infty(X;\EE)
		\hookrightarrow
		\D'(X;\EE)
		\hookrightarrow
		\D_s'(X;\EE).
	\end{equation}
	If \(\mathcal U\subset X\) is open, restriction defines a continuous
	linear map
	\[
	\D_s'(X;\EE)
	\longrightarrow
	\D_s'(\mathcal U;\EE|_{\mathcal U}).
	\]
	The support of an element of \(\D_s'(X;\EE)\) is defined in the usual
	way through these restrictions.
\end{proposition}

\begin{proof}
	Since \(\G_c^s(X;\EE^*)\) is a DFS space, its strong dual is an FS
	space. The first two maps in \eqref{eq:natural-inclusions} are the
	standard continuous embeddings. The last one is obtained by
	transposing the continuous dense inclusion
	\(\G_c^s(X;\EE^*) \hookrightarrow C_c^\infty(X;\EE^*),\) 
	which is dense because \(s>1\).
	
	If \(\mathcal U\subset X\) is open, extension by zero defines a
	continuous map
	\[
	\G_c^s(\mathcal U;\EE^*|_{\mathcal U})
	\longrightarrow
	\G_c^s(X;\EE^*).
	\]
	
	Its transpose is the restriction map on Gevrey ultradistributions.
\end{proof}

\begin{remark}\label{rem:density-bundle}
	The use of the fixed density \(\mu\) is only a notational
	convenience. An invariant definition is
	\[
	\D_s'(X;\EE)
	=
	\left(
	\G_c^s
	\bigl(X;\EE^*\otimes\mathcal D_X\bigr)
	\right)'_b.
	\]
	
	A positive real-analytic density identifies this space topologically
	with the realization in
	\eqref{eq:gevrey-ultradistributions}. Different choices of positive
	real-analytic densities give canonically isomorphic realizations,
	since the ratio of any two such densities is a positive
	real-analytic function.
	
	For the invariant construction of distributions with values in vector
	bundles, see \cite[Section~3.1]{GrosserEtAl2001}. The Gevrey
	formulation is obtained by replacing compactly supported smooth
	sections with compactly supported Gevrey sections.
\end{remark}

\subsection{Differential operators and transposition} \

Let
\[
P:C^\infty(X;\EE)\longrightarrow C^\infty(X;\FF)
\]
be a linear differential operator whose coefficients are of Gevrey order
\(s\) in analytic local coordinates and analytic local frames. Since
Gevrey classes are stable under differentiation and multiplication, \(P\)
induces continuous linear maps
\begin{equation}\label{eq:P-on-gevrey}
	P:\G^s(X;\EE)\longrightarrow\G^s(X;\FF),
	\qquad
	P:\G_c^s(X;\EE)\longrightarrow\G_c^s(X;\FF).
\end{equation}
The second map is well defined because differential operators do not
increase supports.

Let \(\mu\) be the positive real-analytic density fixed above. The formal
transpose of \(P\), with respect to \(\mu\) and the fiberwise duality
pairings, is the differential operator
\[
{}^tP:C^\infty(X;\FF^*)\longrightarrow C^\infty(X;\EE^*)
\]
characterized by
\begin{equation}\label{eq:formal-transpose}
	\int_X\langle Pu,\varphi\rangle\,\mu
	=
	\int_X\langle u,{}^tP\varphi\rangle\,\mu,
\end{equation}
for every \(u\in C^\infty(X;\EE)\) and
\(\varphi\in C_c^\infty(X;\FF^*)\).

The coefficients of \({}^tP\) are again of Gevrey order \(s\). Indeed, in
local coordinates they are obtained from the coefficients of \(P\) and
the density \(\mu\) by differentiation and multiplication. Consequently,
\({}^tP\) defines a continuous linear map
\begin{equation}\label{eq:transpose-on-test-sections}
	{}^tP:
	\G_c^s(X;\FF^*)
	\longrightarrow
	\G_c^s(X;\EE^*).
\end{equation}

The strong transpose of \eqref{eq:transpose-on-test-sections} defines a
continuous operator
\begin{equation}\label{eq:P-on-ultradistributions-map}
	P:\D_s'(X;\EE)\longrightarrow\D_s'(X;\FF)
\end{equation}
by
\begin{equation}\label{eq:P-on-ultradistributions}
	\langle Pu,\varphi\rangle
	:=
	\langle u,{}^tP\varphi\rangle,
	\qquad
	u\in\D_s'(X;\EE),\quad
	\varphi\in\G_c^s(X;\FF^*).
\end{equation}
This extension agrees with the usual action of \(P\) on smooth sections
and distributions.

We use the same notation for a differential operator and its action on
Gevrey sections and ultradistributions. The same constructions apply on
every open subset of \(X\) and commute with restrictions.

\section{The tube complex on \(M\times\T^m\)}
\label{sec:tube-complex}

Let \(M\) be a connected, oriented real-analytic manifold of dimension
\(n\), let \(\T^m:=\R^m/(2\pi\Z^m)\), and set
\[
X:=M\times\T^m.
\]

We denote by \(\pi_M:X\longrightarrow M\) the canonical projection.

For each \(q=0,\ldots,n\), consider the real-analytic complex vector
bundle
\[
\df^{0,q}X
:=
\pi_M^*\bigl(\df^q\C T^*M\bigr).
\]

Its smooth sections are precisely the complex-valued differential forms
on \(X\) involving only covectors in the \(M\)-directions. Thus, if
\(t=(t_1,\ldots,t_n)\) are local coordinates on \(M\), every
\(u\in C^\infty(X;\df^{0,q}X)\) can be written locally as
\[
u(t,x)
=
\sum_{|J|=q}u_J(t,x)\,\mathrm dt_J,
\qquad x\in\T^m.
\]

	We use the notation
\begin{equation}\label{eq:forms-as-sections}
	\df^{0,q}\F(X)
	:=
	\F(X;\df^{0,q}X),
	\qquad
	\df^{0,0}\F(X)=\F(X),
\end{equation}
where $\F\in \{\G^s,\G^s_c,\D_s',C^\infty,C_c^\infty\}$.

Let \(\omega_1,\ldots,\omega_m\in\df^1\G^s(M)\) 
be real-valued closed \(1\)-forms. We identify each \(\omega_k\) 
with its pullback to \(X\). For \(q=0,\ldots,n-1\), define
\begin{equation}\label{eq:tube-complex-operator}
	\L^q:
	\df^{0,q}C^\infty(X)
	\longrightarrow
	\df^{0,q+1}C^\infty(X)
\end{equation}
by
\[
\L^q u
=
\mathrm d_tu
+
\sum_{k=1}^m
\omega_k\wedge\partial_{x_k}u,
\]
where \(\mathrm d_t\) denotes the exterior derivative in the
\(M\)-variable and \(\partial_{x_k}\) acts coefficientwise.

Since the forms \(\omega_k\) are closed and the operators
\(\partial_{x_k}\) commute, a direct computation gives
\[
\L^{q+1}\L^q=0.
\]
Hence,
\[
0\longrightarrow C^\infty(X)
\xrightarrow{\L^0}
\df^{0,1}C^\infty(X)
\xrightarrow{\L^1}\cdots
\xrightarrow{\L^{n-1}}
\df^{0,n}C^\infty(X)
\longrightarrow0
\]
is a differential complex.

Since the coefficients of the operators \(\L^q\) are of Gevrey order
\(s\), the same operators define continuous maps
\[
\L^q:
\df^{0,q}\G^s(X)
\longrightarrow
\df^{0,q+1}\G^s(X)
\]
and
\[
\L^q:
\df^{0,q}\G_c^s(X)
\longrightarrow
\df^{0,q+1}\G_c^s(X).
\]

By transposition, they also act continuously on the corresponding spaces
of Gevrey ultradistributions.

Assume from now on that \(M\) is oriented, and let \(\mathrm dx\) denote
the normalized Haar density on \(\T^m\). For complementary degrees,
consider the bilinear pairing
\[
\langle u,v\rangle
:=
\int_{\T^m}
\left(
\int_M u\wedge v
\right)\mathrm dx,
\]
whenever at least one of the factors has compact support in the
\(M\)-variable. 

If
\[
u\in\df^{0,q}C^\infty(X),
\qquad
v\in\df^{0,n-q-1}C_c^\infty(X),
\]
then integration by parts on \(M\) and on \(\T^m\) gives
\begin{equation}\label{eq:integration-by-parts-Lq}
	\langle\L^qu,v\rangle
	=
	(-1)^{q+1}
	\langle u,\L^{n-q-1}v\rangle.
\end{equation}
Consequently,
\begin{equation}\label{eq:transpose-Lq}
	{}^t\L^q
	=
	(-1)^{q+1}\L^{n-q-1}.
\end{equation}
In particular,
\[
{}^t\L^0=-\L^{n-1},
\qquad
{}^t\L^{n-1}=(-1)^n\L^0.
\]
The signs will play no role in the solvability arguments below.

\begin{remark}\label{rem:nonorientable-base}
	On a non-orientable manifold, the preceding pairing can be formulated
	using forms with values in the orientation bundle, or equivalently
	suitable density-valued forms. The transpose then acts on the
	corresponding twisted complex. We assume that \(M\) is oriented in
	order to avoid this additional notation.
\end{remark}

\subsection{Solvability notions and an abstract criterion} \

Throughout this subsection, \(X\) is assumed to be non-compact. Let
\[
P:\D_s'(X;\EE)\longrightarrow\D_s'(X;\FF)
\]
be a continuous local operator. Here, locality means that \(P\) acts on
every open subset \(\mathcal U\subset X\) and commutes with restrictions:
whenever \(\mathcal V\subset\mathcal U\),
\[
P_{\mathcal V}\bigl(u|_{\mathcal V}\bigr)
=
(P_{\mathcal U}u)|_{\mathcal V}.
\]

We also assume that, on every open subset \(\mathcal U\subset X\), its
transpose defines a continuous map
\[
{}^tP:
\G_c^s(\mathcal U;\FF^*)
\longrightarrow
\G_c^s(\mathcal U;\EE^*).
\]

Every differential operator with coefficients of Gevrey order \(s\)
considered below satisfies these assumptions.

\begin{definition}\label{def:non-confinement}
	We say that \({}^tP\) has the \emph{property of non-confinement of
		Gevrey singularities} if
	\[
	u\in C_c^\infty(X;\FF^*),
	\qquad
	{}^tPu\in\G_c^s(X;\EE^*)
	\quad\Longrightarrow\quad
	u\in\G_c^s(X;\FF^*).
	\]
\end{definition}

\begin{definition}\label{def:P-convex-supports}
	We say that \(X\) is \emph{\(P\)-convex with respect to supports} if,
	for every compact set \(K\subset X\), there exists a compact set
	\(K'\subset X\) such that
	\[
	\varphi\in\G_c^s(X;\FF^*),
	\qquad
	\operatorname{supp}({}^tP\varphi)\subset K
	\quad\Longrightarrow\quad
	\operatorname{supp}\varphi\subset K'.
	\]
\end{definition}

\begin{definition}\label{def:solvability-notions}
	We say that \(P\) is:
	\begin{enumerate}[(i)]
		\item
		\emph{semiglobally solvable in Gevrey ultradistributions} if,
		for every \(f\in\D_s'(X;\FF)\) and every relatively compact open
		set \(\mathcal W\Subset X\), there exists
		\(u\in\D_s'(\mathcal W;\EE)\) such that
		\[
		Pu=f|_{\mathcal W};
		\]
		
		\item
		\emph{globally solvable in Gevrey ultradistributions} if, for
		every \(f\in\D_s'(X;\FF)\), there exists
		\(u\in\D_s'(X;\EE)\) such that
		\[
		Pu=f.
		\]
		Equivalently, \(P\) is globally solvable if it is surjective.
	\end{enumerate}
\end{definition}

Clearly, global solvability implies semiglobal solvability. The following
criterion reduces global solvability to injectivity, regularity, and support
control for the transpose.

\begin{proposition}[A criterion for global solvability]
	\label{prop:abstract-solvability}
	Assume that
	\begin{enumerate}[(i)]
		\item
		\(	{}^tP: \G_c^s(X;\FF^*)  \longrightarrow \G_c^s(X;\EE^*)\) is injective;
		
		\item
		\({}^tP\) has the property of non-confinement of Gevrey singularities;
		
		\item
		\(X\) is \(P\)-convex with respect to supports. 
	\end{enumerate}
	Then \(P\) is globally solvable in Gevrey ultradistributions.
\end{proposition}

\begin{proof}
	By \cite[Theorem~4.16]{Araujo2017}, assumptions \emph{(i)} and
	\emph{(ii)} imply that \(P\) is semiglobally solvable in Gevrey
	ultradistributions. Assumption \emph{(iii)} then allows us to apply
	\cite[Theorem~4.20]{Araujo2017}, which promotes semiglobal solvability
	to global solvability.
\end{proof}

\section{Global solvability of the tube complex}
\label{sec:global-solvability}

We retain the notation and assumptions of
Section~\ref{sec:tube-complex} and assume from now on that \(M\) is
non-compact. Thus,
\[
0\longrightarrow \G^s(X)
\xrightarrow{\L^0}
\df^{0,1}\G^s(X)
\xrightarrow{\L^1}\cdots
\xrightarrow{\L^{n-1}}
\df^{0,n}\G^s(X)
\longrightarrow 0
\]
is the Gevrey complex associated with
\(\boldsymbol{\omega}=(\omega_1,\ldots,\omega_m)\).

The fiber-translation and transport arguments used below may be viewed as
a physical-space reformulation of the partial Fourier series approach to
global solvability and global hypoellipticity for tube structures; see
\cite{BCM1993,BCP1996,BP1999_jmaa,ADL2023gh,ADL2023gs,
	HZ2017,HZ2019,CKMT}.

\subsection{Fiber translations and local normal forms} \

The local normal form of the complex is obtained by translating the torus
fibers. Let \(\mathcal U\subset M\) be open and let
\[
\boldsymbol{\phi}
=
(\phi_1,\ldots,\phi_m):
\mathcal U\longrightarrow\R^m
\]
be a Gevrey map of order \(s\). Define
\[
\tau_{\boldsymbol{\phi}}:
\mathcal U\times\T^m
\longrightarrow
\mathcal U\times\T^m,
\qquad
\tau_{\boldsymbol{\phi}}(t,x)
=
(t,x+\boldsymbol{\phi}(t)),
\]
where the second component is understood modulo \(2\pi\Z^m\).

For \(0\leq q\leq n\), define
\[
\mathcal T_{\boldsymbol{\phi}}^q:
\df^{0,q}C^\infty(\mathcal U\times\T^m)
\longrightarrow
\df^{0,q}C^\infty(\mathcal U\times\T^m)
\]
by
\[
\mathcal T_{\boldsymbol{\phi}}^q
\left(
\sum_{|I|=q}u_I(t,x)\,\mathrm dt_I
\right)
=
\sum_{|I|=q}
u_I\bigl(t,x+\boldsymbol{\phi}(t)\bigr)\,\mathrm dt_I.
\]

Since \(\tau_{\boldsymbol{\phi}}\) is a Gevrey diffeomorphism with inverse
\(\tau_{-\boldsymbol{\phi}}\), the operator
\(\mathcal T_{\boldsymbol{\phi}}^q\) is a topological automorphism of
\(\df^{0,q}\G^s(\mathcal U\times\T^m)\) and of the corresponding
compactly supported space. By transposition, it also induces a
topological automorphism of
\(\df^{0,q}\D_s'(\mathcal U\times\T^m)\).

\begin{proposition}[Covariance under fiber translations]
	\label{prop:gauge-covariance}
	Let \(\boldsymbol{\eta} = (\eta_1,\ldots,\eta_m)\) 
	be a family of Gevrey \(1\)-forms of order \(s\) on \(\mathcal U\), and
	write
	\[
	\boldsymbol{\eta}+\mathrm d\boldsymbol{\phi}
	:=
	(\eta_1+\mathrm d\phi_1,\ldots,
	\eta_m+\mathrm d\phi_m).
	\]
	Then, for \(q=0,\ldots,n-1\),
	\begin{equation}\label{eq:gauge-covariance}
		\L_{\boldsymbol{\eta}}^q
		\mathcal T_{\boldsymbol{\phi}}^q
		=
		\mathcal T_{\boldsymbol{\phi}}^{q+1}
		\L_{\boldsymbol{\eta}+\mathrm d\boldsymbol{\phi}}^q.
	\end{equation}
	
	In particular, if
	\(\omega_k|_{\mathcal U}=\mathrm d\phi_k\), \(k=1,\dots,m\), then
	\begin{equation}\label{eq:local-normal-form}
		\L_{\boldsymbol{\omega}}^q
		=
		(\mathcal T_{\boldsymbol{\phi}}^{q+1})^{-1}
		\mathrm d_t
		\mathcal T_{\boldsymbol{\phi}}^q
		\qquad
		\text{on }\mathcal U\times\T^m.
	\end{equation}
\end{proposition}

\begin{proof}
	Let 
	\[
	u
	=
	\sum_{|I|=q}u_I(t,x)\,\mathrm dt_I
	\in
	\df^{0,q}C^\infty(\mathcal U\times\T^m).
	\]
	By the chain rule,
	\[
	\mathrm d_t
	\bigl(
	u_I(t,x+\boldsymbol{\phi}(t))
	\bigr)
	=
	(\mathrm d_tu_I)
	(t,x+\boldsymbol{\phi}(t))
	+
	\sum_{k=1}^m
	(\partial_{x_k}u_I)
	(t,x+\boldsymbol{\phi}(t))
	\,\mathrm d\phi_k(t).
	\]

	Summing over \(I\), we obtain
	\[
	\mathrm d_t
	\bigl(
	\mathcal T_{\boldsymbol{\phi}}^q u
	\bigr)
	=
	\mathcal T_{\boldsymbol{\phi}}^{q+1}
	\left(
	\mathrm d_tu
	+
	\sum_{k=1}^m
	\mathrm d\phi_k\wedge
	\partial_{x_k}u
	\right).
	\]
	
	Moreover,
	\[
	\partial_{x_k} \bigl( \mathcal T_{\boldsymbol{\phi}}^q u \bigr)
	= \mathcal T_{\boldsymbol{\phi}}^q \bigl(\partial_{x_k}u \bigr).
	\]
	
	Since \(\eta_k\) depends only on the base variable,
	\[
	\eta_k\wedge
	\partial_{x_k}
	\bigl(
	\mathcal T_{\boldsymbol{\phi}}^q u
	\bigr)
	=
	\mathcal T_{\boldsymbol{\phi}}^{q+1}
	\bigl(
	\eta_k\wedge\partial_{x_k}u
	\bigr).
	\]
	Combining these identities gives
	\[
	\L_{\boldsymbol{\eta}}^q
	\mathcal T_{\boldsymbol{\phi}}^q u
	=
	\mathcal T_{\boldsymbol{\phi}}^{q+1}
	\left[
	\mathrm d_tu
	+
	\sum_{k=1}^m
	(\eta_k+\mathrm d\phi_k)
	\wedge\partial_{x_k}u
	\right],
	\]
	which proves \eqref{eq:gauge-covariance}.
	
	If
	\(\boldsymbol{\omega}=\mathrm d\boldsymbol{\phi}\) on
	\(\mathcal U\), taking \(\boldsymbol{\eta}=0\) gives
	\[
	\mathrm d_t
	\mathcal T_{\boldsymbol{\phi}}^q
	=
	\mathcal T_{\boldsymbol{\phi}}^{q+1}
	\L_{\boldsymbol{\omega}}^q.
	\]
	Since
	\((\mathcal T_{\boldsymbol{\phi}}^{q+1})^{-1}
	= \mathcal T_{-\boldsymbol{\phi}}^{q+1}\), 
	the local normal form follows.
\end{proof}

A global version of the previous result can be found in \cite[Theorem 7.1]{AFJR2026_arXiv}.

The following elementary lemma complements the local normal form
\eqref{eq:local-normal-form}; see also
\cite[proof of Corollary~3.4]{ADL2023gs}.

\begin{lemma}[Gevrey primitives with a parameter]
	\label{lemma:gevrey-primitive}
	Let \(Q=I_1\times\cdots\times I_n\subset\R^n\) 
	be an open rectangle, and let \(t^0\in Q\). Suppose that
	\[
	v\in C^\infty(Q\times\T^m),
	\qquad
	\mathrm d_tv
	=
	\sum_{j=1}^n g_j(t,x)\,\mathrm dt_j
	\in\df^{0,1}\G^s(Q\times\T^m),
	\]
	and
	\[
	v(t^0,\cdot)\in\G^s(\T^m).
	\]
	Then
	\[
	v\in\G^s(Q\times\T^m).
	\]
\end{lemma}

\begin{proof}
	For \(t\in Q\), integrate \(\mathrm d_tv\) along the coordinate polygonal
	path joining \(t^0\) to \(t\). This gives
	\[
	v(t,x)
	=
	v(t^0,x)
	+
	\sum_{j=1}^n
	\int_{t_j^0}^{t_j}
	g_j(t_1,\ldots,t_{j-1},r,t_{j+1}^0,\ldots,t_n^0,x)
	\,\mathrm dr.
	\]
	
	On every compact rectangle \(Q'\Subset Q\), the integration segments are
	contained in a fixed compact subset of \(Q\). The Gevrey estimates for the
	functions \(g_j\) are therefore uniform on these segments. Differentiating
	the preceding identity under the integral sign, and using the fundamental
	theorem of calculus when derivatives fall on the upper endpoint, yields
	constants \(C,h>0\) such that
	\[
	\sup_{(t,x)\in Q'\times\T^m}
	\left|
	\partial_t^\alpha\partial_x^\beta v(t,x)
	\right|
	\leq
	Ch^{|\alpha|+|\beta|}
	\alpha!^s\beta!^s.
	\]
	Since \(Q'\Subset Q\) is arbitrary, \(v\in\G^s(Q\times\T^m)\).
\end{proof}

The propagation arguments below are based on a transport identity associated
with \(\L^0\). Let \(\gamma:[a,b]\to M\) be a piecewise \(C^1\) path and set
\[
A_\gamma(r)
:=
\left(
\int_a^r
\omega_1|_{\gamma(\rho)}\bigl(\dot\gamma(\rho)\bigr)
\,\mathrm d\rho,
\dots,
\int_a^r
\omega_m|_{\gamma(\rho)}\bigl(\dot\gamma(\rho)\bigr)
\,\mathrm d\rho
\right)
\in\R^m.
\]
For each \(x\in\T^m\), the curve
\[
r\longmapsto
\bigl(\gamma(r),x+A_\gamma(r)\bigr)
\]
is the horizontal lift of \(\gamma\) through \((\gamma(a),x)\) for the
flat connection determined by the forms
\(\mathrm dx_k-\omega_k\). Integrating the derivative of a function along
this lifted path gives the following formula. For general background on
involutive and tube structures, see \cite{BCH_book}.

\begin{proposition}[Transport formula]
	\label{prop:transport-formula}
	Let \(u\in C^1(M\times\T^m)\) and set
	\[
	f:=\L^0u\in\df^{0,1}C^0(M\times\T^m).
	\]
	Then, for every piecewise \(C^1\) path
	\(\gamma:[a,b]\to M\) and every \(x\in\T^m\),
	\begin{equation}\label{eq:transport-formula}
		u\bigl(\gamma(b),x+A_\gamma(b)\bigr)
		-
		u\bigl(\gamma(a),x\bigr)
		=
		\int_a^b
		f_{\left(\gamma(r),\,x+A_\gamma(r)\right)}
		\bigl(\dot\gamma(r)\bigr)\,\mathrm dr.
	\end{equation}
	Here and below, additions in the torus variable are understood modulo
	\(2\pi\Z^m\).
	
	In particular, if \(\L^0u=0\) on  \(\gamma([a,b])\times\T^m\), then
	\begin{equation}\label{eq:homogeneous-transport}
		u\bigl(\gamma(b),x\bigr)
		=
		u\bigl(\gamma(a),x-A_\gamma(b)\bigr).
	\end{equation}
\end{proposition}

\begin{proof}
	Fix \(x\in\T^m\) and define
	\[
	w_x(r)
	:=
	u\bigl(\gamma(r),x+A_\gamma(r)\bigr).
	\]
	On every subinterval on which \(\gamma\) is of class \(C^1\),
	\[
	\frac{\mathrm d}{\mathrm dr}A_\gamma(r)
	=
	\bigl(
	\omega_1|_{\gamma(r)}(\dot\gamma(r)),
	\ldots,
	\omega_m|_{\gamma(r)}(\dot\gamma(r))
	\bigr).
	\]

	Hence, by the chain rule,
	\begin{align*}
		\frac{\mathrm d}{\mathrm dr}w_x(r)
		={}&
		\bigl(\mathrm d_tu\bigr)_{
			\left(\gamma(r),\,x+A_\gamma(r)\right)}
		\bigl(\dot\gamma(r)\bigr)
		\\
		&+
		\sum_{k=1}^m
		\omega_k|_{\gamma(r)}\bigl(\dot\gamma(r)\bigr)
		\bigl(\partial_{x_k}u\bigr)
		\bigl(\gamma(r),x+A_\gamma(r)\bigr)
		\\
		={}&
		f_{\left(\gamma(r),\,x+A_\gamma(r)\right)}
		\bigl(\dot\gamma(r)\bigr).
	\end{align*}

	Integrating on each \(C^1\) subinterval of \([a,b]\) and summing gives
	\[
	w_x(b)-w_x(a)
	=
	\int_a^b
	f_{\left(\gamma(r),\,x+A_\gamma(r)\right)}
	\bigl(\dot\gamma(r)\bigr)\,\mathrm dr.
	\]
	
	Since \(A_\gamma(a)=0\), this proves
	\eqref{eq:transport-formula}.
	
	If \(\L^0u=0\) on
	\(\gamma([a,b])\times\T^m\), the integral vanishes, and therefore
	\[
	u\bigl(\gamma(b),x+A_\gamma(b)\bigr)
	=
	u\bigl(\gamma(a),x\bigr).
	\]
	Replacing \(x\) by \(x-A_\gamma(b)\) yields
	\eqref{eq:homogeneous-transport}.
\end{proof}

\subsection{Compactly supported solutions and non-confinement of
	Gevrey singularities} \

The transport formula first yields injectivity on compactly supported
functions. This property is independent of any arithmetic condition on
\(\boldsymbol{\omega}\).

\begin{proposition}\label{prop:compact-support-injectivity}
	The map
	\[
	\L^0: C_c^\infty(X) \longrightarrow \df^{0,1}C_c^\infty(X)
	\]
	is injective. Consequently,
	\[
	\L^0: \G_c^s(X) \longrightarrow \df^{0,1}\G_c^s(X)
	\]
	is injective.
\end{proposition}

\begin{proof}
	Let \(u\in C_c^\infty(X)\) satisfy \(\L^0u=0\). Since
	\(\pi_M(\operatorname{supp}u)\) is compact and \(M\) is non-compact,
	there exists \(t_0\in M\setminus\pi_M(\operatorname{supp}u)\).
	In particular, \(u(t_0,\cdot)=0\).
	
	Fix \(t\in M\) and choose a piecewise \(C^1\) path \(\gamma\) joining
	\(t_0\) to \(t\). Formula~\eqref{eq:homogeneous-transport} gives
	\[
	u(t,x)
	=
	u\bigl(t_0,x-A_\gamma(b)\bigr)
	=
	0.
	\]
	Since \(t\) is arbitrary, we have \(u=0\). The second assertion follows from
	the inclusion \(\G_c^s(X)\subset C_c^\infty(X)\).
\end{proof}

The same ingredients give a propagation result for Gevrey regularity.

\begin{proposition}[Propagation of Gevrey regularity]
	\label{prop:gevrey-regularity-propagation}
	Let \(u\in C^\infty(X)\) and suppose that
	\[
	f:=\L^0u\in\df^{0,1}\G^s(X).
	\]
	If there exists \(t_0\in M\) such that \(u(t_0,\cdot)\in\G^s(\T^m),\)
	then
	\[
	u\in\G^s(X).
	\]
\end{proposition}

\begin{proof}
	Fix \(t\in M\), and let
	\(\gamma:[a,b]\to M\) be a piecewise \(C^1\) path from \(t_0\) to
	\(t\). Replacing \(x\) by \(x-A_\gamma(b)\) in
	\eqref{eq:transport-formula}, we obtain
	\begin{align}
		u(t,x)
		={}
		u\bigl(t_0,x-A_\gamma(b)\bigr)
		+
		\int_a^b
		f_{\left(
			\gamma(r),\,
			x-A_\gamma(b)+A_\gamma(r)
			\right)}
		\bigl(\dot\gamma(r)\bigr)\,\mathrm dr.
		\label{eq:fiber-regularity-transport}
	\end{align}
	
	The first term on the right-hand side belongs to \(\G^s(\T^m)\).
	Since \(\gamma([a,b])\) is compact, the Gevrey estimates for \(f\)
	are uniform along the path. Differentiating the integral in
	\eqref{eq:fiber-regularity-transport} with respect to \(x\), we find
	constants \(C,h>0\), depending on \(\gamma\), such that
	\[
	\sup_{x\in\T^m}
	\left|
	\partial_x^\beta u(t,x)
	\right|
	\leq
	C h^{|\beta|}\beta!^s,
	\qquad
	\beta\in\N_0^m.
	\]
	
	Hence,
	\[
	u(t,\cdot)\in\G^s(\T^m).
	\]
	
	Choose an analytic coordinate neighborhood
	\(\mathcal U\subset M\) of \(t\) whose coordinate image is an open
	rectangle. After shrinking \(\mathcal U\), the closed forms
	\(\omega_k\) admit smooth potentials
	\[
	\omega_k|_{\mathcal U}
	=
	\mathrm d\phi_k,
	\qquad k=1,\ldots,m.
	\]

	Since \(\mathrm d\phi_k\in\df^1\G^s(\mathcal U)\), the scalar case of
	Lemma~\ref{lemma:gevrey-primitive} implies that
	\(\phi_k\in\G^s(\mathcal U)\).

	Set
	\[
	\boldsymbol{\phi}
	:=
	(\phi_1,\ldots,\phi_m),
	\qquad
	v:=\mathcal T_{\boldsymbol{\phi}}^0u,
	\qquad
	g:=\mathcal T_{\boldsymbol{\phi}}^1f.
	\]
	
	By Proposition~\ref{prop:gauge-covariance},
	\[
	\mathrm d_tv
	=
	g
	\in
	\df^{0,1}\G^s(\mathcal U\times\T^m).
	\]
	
	Moreover,
	\[
	v(t,x)
	=
	u\bigl(t,x+\boldsymbol{\phi}(t)\bigr)
	\in\G^s(\T^m).
	\]
	
	Lemma~\ref{lemma:gevrey-primitive} therefore gives
	\(v\in\G^s(\mathcal U\times\T^m)\). 
	Since
	\((\mathcal T_{\boldsymbol{\phi}}^0)^{-1}\) preserves the Gevrey
	class,
	\[
	u\in\G^s(\mathcal U\times\T^m).
	\]
	
	As \(t\in M\) was arbitrary, \(u\in\G^s(X)\).
\end{proof}

\begin{corollary}\label{prop:nonconfinement-L0}
	The operator \(\L^0\) has the property of non-confinement of Gevrey
	singularities. More precisely,
	\[
	u\in C_c^\infty(X),
	\qquad
	\L^0u\in\df^{0,1}\G_c^s(X)
	\quad\Longrightarrow\quad
	u\in\G_c^s(X).
	\]
\end{corollary}

\begin{proof}
	Since \(M\) is non-compact and
	\(\pi_M(\operatorname{supp}u)\) is compact, there exists
	\[
	t_0\in M\setminus\pi_M(\operatorname{supp}u).
	\]
	
	Thus \(u(t_0,\cdot)=0\), and
	Proposition~\ref{prop:gevrey-regularity-propagation} gives
	\(u\in\G^s(X)\). Since \(u\) has compact support,
	\(u\in\G_c^s(X)\).
\end{proof}

\begin{remark}\label{rem:noncompactness-seed}
	The two geometric hypotheses play different roles in the preceding
	argument. The non-compactness of \(M\) provides a fiber on which a
	compactly supported function vanishes, while the closedness of the
	forms \(\omega_k\) provides local Gevrey potentials and hence the
	normal form \eqref{eq:local-normal-form}. No arithmetic property of
	the periods of the forms is used.
\end{remark}

\begin{corollary}\label{cor:semiglobal-solvability}
	The operator
	\[
	\L^{n-1}:
	\df^{0,n-1}\D_s'(X)
	\longrightarrow
	\df^{0,n}\D_s'(X)
	\]
	is semiglobally solvable in Gevrey ultradistributions.
\end{corollary}

\begin{proof}
	By \eqref{eq:transpose-Lq},
	\[
	{}^t\L^{n-1}
	=
	(-1)^n\L^0.
	\]
	Proposition~\ref{prop:compact-support-injectivity} gives the
	injectivity of the transpose on compactly supported Gevrey functions,
	while Corollary~\ref{prop:nonconfinement-L0} gives the
	non-confinement of Gevrey singularities. The conclusion follows from
	\cite[Theorem~4.16]{Araujo2017}.
\end{proof}

\subsection{Propagation of supports} \

To pass from semiglobal to global solvability, it remains to verify the
support condition in Proposition~\ref{prop:abstract-solvability}. We first
construct a compact neighborhood with no relatively compact complementary
components.

\begin{lemma}\label{lemma:regular-compact-neighborhood}
	For every compact set \(K\subset M\), there exists a compact set
	\(K^\sharp\subset M\) such that
	\[
	K\subset\operatorname{int}K^\sharp
	\]
	and every connected component of \(M\setminus K^\sharp\) is
	non-compact.
\end{lemma}

\begin{proof}
	Choose a relatively compact connected open neighborhood
	\(\Omega\Subset M\) of \(K\) with smooth boundary. Since
	\(\partial\Omega\) is compact, it has finitely many connected
	components. Moreover, every connected component of
	\(M\setminus\overline\Omega\) has closure meeting
	\(\partial\Omega\), because \(M\) is connected. It follows from a
	collar neighborhood of \(\partial\Omega\) that
	\(M\setminus\overline\Omega\) has only finitely many connected
	components.
	
	Let \(\mathcal U_1,\ldots,\mathcal U_N\) be its relatively compact
	connected components and set
	\[
	K^\sharp
	:=
	\overline\Omega
	\cup
	\bigcup_{j=1}^N\overline{\mathcal U_j}.
	\]
	
	Then \(K^\sharp\) is compact and contains \(K\) in its interior. Its
	complement consists precisely of the non-relatively compact components
	of \(M\setminus\overline\Omega\), and therefore every connected
	component of \(M\setminus K^\sharp\) is non-compact.
\end{proof}

\begin{proposition}\label{prop:support-propagation-L0}
	Let \(\mathcal K\subset X\) be compact, set \(K:=\pi_M(\mathcal K)\),
	and let \(K^\sharp\) be as in
	Lemma~\ref{lemma:regular-compact-neighborhood}. If
	\(u\in C_c^\infty(X)\) and 
	\(\operatorname{supp}(\L^0u)\subset\mathcal K\), 
	then
	\begin{equation}\label{eq:support-hull}
		\operatorname{supp}u
		\subset
		K^\sharp\times\T^m.
	\end{equation}
	Consequently, \(X\) is \(\L^{n-1}\)-convex with respect to supports.
\end{proposition}

\begin{proof}
	Let \(\mathcal U\) be a connected component of
	\(M\setminus K^\sharp\). By
	Lemma~\ref{lemma:regular-compact-neighborhood}, the set
	\(\mathcal U\) is non-compact. Since
	\(\pi_M(\operatorname{supp}u)\) is compact, there exists
	\[
	t_{\mathcal U}
	\in
	\mathcal U\setminus\pi_M(\operatorname{supp}u).
	\]
	In particular, \(u(t_{\mathcal U},\cdot)=0\).
	
	Since \(K\subset K^\sharp\), one has
	\[
	\L^0u=0
	\quad\text{on }\mathcal U\times\T^m.
	\]
	
	For any \(t\in\mathcal U\), choose a piecewise \(C^1\) path in
	\(\mathcal U\) joining \(t_{\mathcal U}\) to \(t\). The homogeneous
	transport formula \eqref{eq:homogeneous-transport} gives
	\[
	u(t,\cdot)=0.
	\]
	Thus \(u\) vanishes on
	\(\mathcal U\times\T^m\).
	
	This holds for every connected component of
	\(M\setminus K^\sharp\), and hence
	\[
	\operatorname{supp}u
	\subset
	K^\sharp\times\T^m.
	\]
	
	Now let \(P=\L^{n-1}\). By \eqref{eq:transpose-Lq},
	\[
	{}^tP=(-1)^n\L^0.
	\]
	
	For each compact set \(\mathcal K\subset X\), the set
	\[
	\mathcal K'
	:=
	K^\sharp\times\T^m
	\]
	is compact and, by the estimate just proved,
	\[
	\varphi\in\G_c^s(X),
	\qquad
	\operatorname{supp}({}^tP\varphi)\subset\mathcal K
	\quad\Longrightarrow\quad
	\operatorname{supp}\varphi\subset\mathcal K'.
	\]
	
	Therefore, \(X\) is \(P\)-convex with respect to supports.
\end{proof}

\subsection{Global solvability in top degree} \

The preceding results yield the main theorem.

\begin{theorem}\label{thm:global-solvability-top-degree}
	Let \(s>1\), let \(M\) be a connected, non-compact, oriented
	real-analytic manifold of dimension \(n\), and let
	\(\omega_1,\ldots,\omega_m\in\df^1\G^s(M)\)
	be real-valued closed \(1\)-forms. Then
	\begin{equation}\label{eq:global-solvability-top}
		\L^{n-1}:
		\df^{0,n-1}\D_s'(M\times\T^m)
		\longrightarrow
		\df^{0,n}\D_s'(M\times\T^m)
	\end{equation}
	is globally solvable in Gevrey ultradistributions.
\end{theorem}

\begin{proof}
	Apply Proposition~\ref{prop:abstract-solvability} to
	\[
	P=\L^{n-1}.
	\]
	By \eqref{eq:transpose-Lq},
	\[
	{}^tP=(-1)^n\L^0.
	\]
	Proposition~\ref{prop:compact-support-injectivity} gives the
	injectivity of \({}^tP\) on compactly supported Gevrey functions,
	Corollary~\ref{prop:nonconfinement-L0} gives the non-confinement of
	Gevrey singularities, and
	Proposition~\ref{prop:support-propagation-L0} gives
	\(P\)-convexity with respect to supports. Hence all the hypotheses of
	Proposition~\ref{prop:abstract-solvability} are satisfied, and the
	conclusion follows.
\end{proof}

\begin{remark}[Regularity versus solvability]
	\label{rem:regularity-vs-solvability}
	Theorem~\ref{thm:global-solvability-top-degree} imposes no arithmetic
	condition on the periods of the forms
	\(\omega_1,\ldots,\omega_m\). Such conditions may govern global
	Gevrey hypoellipticity of \(\L^0\), but they do not obstruct the
	global solvability of \(\L^{n-1}\) in Gevrey ultradistributions.
	This contrasts with the compact setting discussed in the next
	subsection.
\end{remark}

\begin{remark}[Why top degree is special]\label{rem:top-degree}
	For \(q<n-1\), the complex relation
	\[
	\L^{q+1}\L^q=0
	\]
	implies that
	\[
	\operatorname{Ran}\L^q
	\subset
	\ker\L^{q+1}.
	\]
	
	Thus, in intermediate degree, one cannot expect \(\L^q\) to be
	surjective onto the entire space of \((0,q+1)\)-forms. The natural
	question is whether every \(\L^{q+1}\)-closed datum belongs to the
	range of \(\L^q\), or equivalently whether
	\[
	H_{\L,s}^{q+1}(X)
	:=
	\frac{
		\ker\!\left(
		\L^{q+1}:
		\df^{0,q+1}\D_s'(X)
		\longrightarrow
		\df^{0,q+2}\D_s'(X)
		\right)
	}{
		\operatorname{Ran}\!\left(
		\L^q:
		\df^{0,q}\D_s'(X)
		\longrightarrow
		\df^{0,q+1}\D_s'(X)
		\right)
	}
	\]
	vanishes. In top degree, the differential compatibility condition is
	automatic because \(\L^n=0\). Therefore,
	Theorem~\ref{thm:global-solvability-top-degree} is equivalently the
	statement
	\[
	H_{\L,s}^{n}(X)=0.
	\]
\end{remark}

\subsection{A comparison with the compact case} \

For comparison, assume throughout this subsection that \(M\) is compact
and that \(\omega_1,\ldots,\omega_m\) are real-analytic, while retaining
the remaining hypotheses and notation introduced above.
Then \(X=M\times\T^m\) is compact and
\[
\G_c^s(X)=\G^s(X).
\]

The constant function \(1\) belongs to \(\ker\L^0\). Moreover, for every
\(u\in\df^{0,n-1}\D_s'(X)\),
\[
\bigl\langle \L^{n-1}u,1\bigr\rangle
=
(-1)^n\bigl\langle u,\L^01\bigr\rangle
=
0.
\]

Thus, unlike the non-compact case, \(\L^{n-1}\) cannot be surjective onto
the whole space \(\df^{0,n}\D_s'(X)\). One must restrict the right-hand
side to the annihilator of \(\ker\L^0\).

Following \cite[Definition~1.1]{ADL2023gs}, we adopt the following
compact-case notion of solvability.

\begin{definition}\label{def:compact-global-solvability}
	We say that
	\[
	\L^{n-1}:
	\df^{0,n-1}\D_s'(X)
	\longrightarrow
	\df^{0,n}\D_s'(X)
	\]
	is \emph{globally \((\D_s',\D_s')\)-solvable} if, for every
	\(f\in\df^{0,n}\D_s'(X)\) satisfying
	\[
	\langle f,\varphi\rangle=0, \quad \forall \varphi\in \G^s(M\times\T^m)\cap\ker\L^0,
	\]
	there exists \(u\in\df^{0,n-1}\D_s'(X)\) such that
	\[
	\L^{n-1}u=f.
	\]
\end{definition}

This notion differs from the unrestricted global solvability considered
in Theorem~\ref{thm:global-solvability-top-degree}, since it incorporates
the compatibility conditions forced by the kernel of the transpose.

\begin{proposition}\label{prop:compact-closed-range}
	The following statements are equivalent:
	\begin{enumerate}[(i)]
		\item
		\(\L^{n-1}\) is globally
		\((\D_s',\D_s')\)-solvable;
		
		\item
		\(\L^{n-1}: \df^{0,n-1}\D_s'(X) \longrightarrow \df^{0,n}\D_s'(X)\)
		has closed range;
		
		\item
		\(\L^0: \G^s(X) \longrightarrow \df^{0,1}\G^s(X)\)
		has closed range.
	\end{enumerate}
\end{proposition}

\begin{proof}
	In the present DFS--FS duality, the transpose relation
	\eqref{eq:transpose-Lq} gives
	\[
	\overline{\operatorname{Ran}\L^{n-1}}
	=
	(\ker\L^0)^\perp,
	\]
	where the closure is taken in the strong topology of
	\(\df^{0,n}\D_s'(X)\). Hence, global
	\((\D_s',\D_s')\)-solvability is equivalent to the closedness of the
	range of \(\L^{n-1}\).
	
	Since \(\L^{n-1}\) is, up to sign, the strong transpose of \(\L^0\),
	the equivalence between \emph{(ii)} and \emph{(iii)} follows from the
	closed-range theorem for the corresponding DFS--FS dual pair; see
	\cite[Section~3, in particular (3.3)--(3.4)]{ADL2023gs}.
\end{proof}

We now recall the arithmetic condition that characterizes this
closed-range property. A real closed \(1\)-form
\(\alpha\in\df^1\G^s(M)\) is said to be \emph{integral} if
\[
\frac{1}{2\pi}\int_\sigma\alpha\in\Z
\]
for every \(1\)-cycle \(\sigma\) in \(M\).

Given \(\boldsymbol{\omega} = (\omega_1,\ldots,\omega_m),\)
set
\[
\Gamma_{\boldsymbol{\omega}}
:=
\left\{
\xi\in\Z^m:
\xi\cdot\boldsymbol{\omega}
\text{ is not integral}
\right\}.
\]

\begin{definition}\label{def:relative-exponential-liouville}
	We say that \(\boldsymbol{\omega}\) is
	\emph{\((s,\Gamma_{\boldsymbol{\omega}})\)-exponential Liouville} if there exist
	\(\varepsilon>0\), a sequence
	\[
	\{\theta_j\}_{j\in\N}
	\subset\df^1\G^s(M;\R)
	\]
	of integral \(1\)-forms, and a sequence
	\[
	\{\xi^{(j)}\}_{j\in\N}
	\subset \Gamma_{\boldsymbol{\omega}}
	\]
	with \(|\xi^{(j)}|\to\infty\), such that
	\[
	\left\{
	e^{\varepsilon|\xi^{(j)}|^{1/s}}
	\bigl(
	\xi^{(j)}\cdot\boldsymbol{\omega}
	-
	\theta_j
	\bigr)
	\right\}_{j\in\N}
	\]
	is bounded in \(\df^1\G^s(M)\).
\end{definition}

The following characterization is
\cite[Theorem~1.3]{ADL2023gs}.

\begin{theorem}\label{thm:compact-arithmetic-characterization}
	The operator \(\L^{n-1}\) is globally
	\((\D_s',\D_s')\)-solvable if and only if
	\(\boldsymbol{\omega}\) is not
	\((s,\Gamma_{\boldsymbol{\omega}})\)-exponential Liouville.
\end{theorem}

Thus, in the compact setting, unrestricted global solvability is
impossible, while solvability for compatible data may still fail because
of exponentially small denominators. In contrast,
Theorem~\ref{thm:global-solvability-top-degree} gives unrestricted global
solvability on a non-compact base without any arithmetic condition on the
periods of the defining forms.

	\section*{Acknowledgments} 

	This study was financed in part by CAPES -- Brasil (Finance Code 001). The first author was supported in part by the Italian Ministry of the University and Research -- MUR, within the framework of the Call relating to the scrolling of the final rankings of the PRIN 2022 -- Project Code 2022HCLAZ8, CUP D53C24003370006 (PI A.~Palmieri, Local unit Sc.~Resp.~S.~Coriasco). 
	The second and third authors were supported in part by CNPq -- Brasil (grants 316850/2021-7 and 301573/2025-5, respectively).

	\bibliographystyle{plain}
	\bibliography{references}
	
\end{document}